\pdfoutput=1
\documentclass{amsart}
\usepackage{todonotes}
\usepackage{xypic}
\usepackage{graphicx, amsmath, amssymb, amsthm}
\usepackage{hyperref} 

\newcommand{\support}[1]{The author was supported by {#1}.}

\newcommand{\NSFTwo}{NSF Grant DMS-1307896}

\newcommand{\Sloan}{the Sloan Foundation}

\newcommand{\MSRI}{This material is based on work supported by the National Science Foundation under Grant No. 0932078-000 while the author was in residence at the Mathematical Sciences Research Institute in Berkeley, California, during the Spring 2014 semester.}

\newcommand{\Z}{{\mathbb  Z}}
\newcommand{\F}{{\mathbb F}}
\newcommand{\MUR}{MU_{\mathbb R}}
\newcommand{\MSpin}{MSpin}
\newcommand{\Spin}{Spin}
\newcommand{\XiO}{\Xi O}

\newcommand{\fancydelta}{\mathfrak d}

\newcommand{\Ind}{\big\uparrow} 
\newcommand{\Res}{\big\downarrow} 
\DeclareMathOperator{\Hom}{Hom}
\DeclareMathOperator{\Ext}{Ext}

\newcommand{\res}{res}
\newcommand{\tr}{tr}

\newcommand{\m}[1]{{\protect\underline{#1}}}
\newcommand{\mZ}{\m{\Z}}
\newcommand{\mM}{\m{M}}
  
\newcommand{\mR}{\m{R}}
\newcommand{\mB}{\m{B}}

\newcommand{\mH}{\m{H}}

\newcommand{\mpi}{\m{\pi}}

\mathchardef\mhyphen="2D
 
\newtheorem{theorem}{Theorem}[section]

\newtheorem{lemma}[theorem]{Lemma}
\newtheorem{corollary}[theorem]{Corollary}

\newtheorem{definition}[theorem]{Definition}

\newtheorem{proposition}[theorem]{Proposition}
\newtheorem{prop}[theorem]{Proposition}
\newtheorem{conjecture}[theorem]{Conjecture}

\newtheorem{remark}[theorem]{Remark}
\newtheorem{example}[theorem]{Example}
\newtheorem{notation}[theorem]{Notation}

\newtheorem*{maintheorem}{Main Theorem}
 
\begin{document}
 
\title[On $\eta^{3}$]{On the fate of \texorpdfstring{$\eta^{3}$}{eta cubed} in higher analogues of Real bordism}

\author{Michael~A.~Hill}
\address{Department of Mathematics \\ University of Virginia
\\Charlottesville, VA 22904}
\email{mikehill@virginia.edu}
\thanks{\support{{\NSFTwo} and {\Sloan}}}

\begin{abstract}
We show that the cube of the Hopf map $\eta$ maps to zero under the Hurewicz map for all fixed points of all norms to cyclic $2$-groups of the Landweber-Araki Real bordism spectrum. Using that the slice spectral sequence is a spectral sequence of Mackey functors, we compute the relevant portion of the homotopy groups of these fixed points, showing that multiplication by $4$ annihilates $\pi_{3}$.
\end{abstract}

\date{\today}

\maketitle

\section{Introduction}

Rohlin's theorem is a foundational result in low dimensional topology. The standard formulation is that a smooth, $\Spin$ $4$-manifold has signature divisible by $16$, and Rohlin proved this through a careful analysis of characteristic classes and of the third homotopy groups of spheres \cite{Rohlin}. The Atiyah-Segal index theorem provides a refinement of this, recasting the result in $K$-theory. In particular, a smooth $\Spin$ $4$-manifold with boundary can have a signature divisible by $8$ (which is required by Donaldson's theorem), but in this case, the boundary is $\Spin$-bordant to the three dimensional torus $(S^{1})^{\times 3}$ together with the $\Spin$ structure induced by the Lie group multiplication.

One of the ways to recast this result is via the connection between real $K$-theory, $KO$, and complex $K$-theory, $KU$. Atiyah, using his Real $K$-theory, showed that the natural conjugation action of the cyclic group of order $2$, $C_{2}$ on $KU$ has fixed points $KO$ \cite{AtiyahKR}. Moreover, he showed that the homotopy fixed points of $KU$ and the fixed points of $KU$ agree:
\[
KU^{hC_{2}}\simeq KU^{C_{2}}\simeq KO.
\]
There is an associated spectral sequence, the homotopy fixed point spectral sequence, computing the homotopy groups of $KO$ out of the cohomology of $C_{2}$ acting on the homotopy groups of $KU$:
\[
E_{2}^{s,t}=H^{s}(C_{2};\pi_{t}KU)\Rightarrow \pi_{t-s}KO.
\]
This spectral sequence has a single $d_{3}$ differential:
\[
d_{3}(v_{1}^{2})=\eta^{3}.
\]
The source of this differential represents a smooth $Spin$ $4$-manifold with signature $8$, and the target is the three torus described above. Thus a smooth $Spin$ $4$-manifold of signature $8$ must have boundary. Phrased yet another way, the element $\eta^{3}\in\pi_{3} S^{0}$ must map to zero in $\MSpin$.

The point of this paper is to generalize this result to larger groups. The entire story is again described by geometry: here we have bordism classes of equivariant manifolds for larger cyclic $2$-groups with additional structure on their normal or tangent bundles. If we restrict to the underlying $C_{2}$-equivariant structure, then we recover a Real manifold in the sense of Araki-Landweber \cite{Araki:Coefs, Landweber:MU}. These bordism theories played an essential role in the solution to the Kervaire Invariant One problem, and we will build on the work there \cite{Hill:2009uo}.

Let $C_{2^{n}}$ denote the cyclic group of order $2^{n}$, and let $N_{1}^{n}:=N_{C_{2}}^{C_{2^{n}}}$ denote the Hill-Hopkins-Ravenel norm \cite{Hill:2009uo}. Let $\MUR$ denote the Arak-Landweber Real bordism spectrum, and we define
\[
\Xi_{n}=N_{C_{2}}^{C_{2^{n}}}\MUR.
\]
This is a genuine $C_{2^{n}}$-equivariant spectrum, and we can then take the fixed points:
\[
\XiO_{n}=\Xi_{n}^{C_{2^{n}}}.
\]

Since the norm takes commutative ring spectra to commutative ring spectra and since taking fixed points preserves commutative ring spectra, the spectra $\XiO_{n}$ are all commutative ring spectra. In particular, they have a natural map from the zero sphere.

The homotopy theoretic formulation of our main theorem is as follows

\begin{theorem}
For any $n$, the element $\eta^{3}\in\pi_{3}S^{0}$ maps to zero under the Hurewicz map
\[
\pi_{\ast}S^{0}\to \pi_{\ast} \XiO_{n}.
\]
\end{theorem}

\begin{corollary}
The kernel of the map $\pi_{\ast}S^{0}\to \pi_{\ast}\XiO_{n}$ is large, containing in particular elements like $\eta(\eta\sigma+\epsilon)$ and the elements $P^{i}\eta^{3}$, $i\geq 0$.
\end{corollary}

To prove this, we will actually show a stronger theorem.

\begin{maintheorem}\label{thm:MainTheorem}
 For any $n$, $4$ annihilates the group $\pi_{3}\XiO_{n}$.
\end{maintheorem}

\begin{remark}
Our argument has the same flavor as the $K$-theoretic approach to Rohlin's theorem. We show a series of explicit differentials in the slice spectral sequence (which in a range coincides with the homotopy fixed point spectral sequence), and then we resolve the extensions. There is underlying geometry in all of these differentials, but we do not yet understand it.
\end{remark}

The main theorems have an immediate computational consequence for the Lubin-Tate spectra and the Hopkins-Miller higher real $K$-theories \cite{HoMi, Re97}. If $m=2^{k}(2r+1)$, then the Lubin-Tate spectrum $E_{m}$ is naturally a $C_{2^{k+1}}$-equivariant commutative ring spectrum. The Real orientation of the underlying $C_{2}$-formal group law then gives rise to a $C_{2^{k+1}}$-equivariant orientation map
\[
\Xi_{k+1}\to E_{m}.
\]
Since this is a ring map, we conclude immediately that $\eta^{3}$ is zero in all of these higher real $K$-theories for cyclic groups.

\begin{corollary}
For any $m=2^{k}(2r+1)$, the Hurewicz image of $\eta^{3}$ is zero in
\[
\pi_{3}EO_{m}(C_{2^{k+1}})=\pi_{3}E_{m}^{hC_{2^{k+1}}}.
\]
\end{corollary}

Hewett's classification of the finite subgroups of the Morava stabilizer group \cite{He95} shows that the $2$-torsion subgroup of the $n$\textsuperscript{th} Morava stabilizer group is a cyclic $2$-group if $n\not\equiv 2\mod 4$ and contains the quaternions $Q_{8}$ if $n\equiv 2\mod 4$. Since prime-to-$2$ groups just carve out an eigenspaces for $2$-complete homotopy fixed points, our result applies to all finite subgroups of the Morava stabilizer group not containing the quaternions.

\begin{corollary}
For $m\not\equiv 2\mod 4$, $\eta^{3}$ is not detected in the homotopy fixed points of $E_{m}$ for any finite subgroup of the Morava stabilizer group.
\end{corollary}
When $m\equiv 2\mod 4$, the homotopy fixed points here can detect $\eta^{3}$, as the example of $TMF$ shows \cite{HopkinsMahowald}. We do not know if this is always true.

The proof of Theorem~\ref{thm:MainTheorem} follows from a careful analysis of the slice spectral sequence computing the homotopy groups of $\XiO_{n}$. The $E_{\infty}$ page of the slice spectral sequence, as a spectral sequence of groups, produces a large associated graded $2$-group for the $2$-localization of $\pi_{3}\XiO_{n}$, the maximum torsion order of which is bound only by $2^{n}$. However, equivariant homotopy groups actually sit inside a Mackey functor, and the slice spectral sequence is a spectral sequence of Mackey functors. Remembering this results in an $E_{\infty}$ term which is an associated graded Mackey functor for $\m{\pi}_{3}\Xi_{n}$. Our proof then follows from two smaller results:
\begin{enumerate}
\item there is a predictable extension of Mackey functors which produces $4$-torsion in $\m{\pi}_{3}\Xi_{n}$, and
\item the Mackey functor $\Ext$ groups for all of the other resulting Mackey functors are all zero.
\end{enumerate}

\subsection*{Acknowledgements}

{\MSRI}

\section{Mackey Functors and Representations}
\subsection{Mackey functors and notation}
The heart of computations in equivariant algebraic topology are Mackey functors. These play the role of abelian groups, serving as the coefficients for cohomology and as the target for homotopy functors. There are several equivalent ways to describe these \cite{May:Alaska}. The simplest is that these are additive functors to abelian groups from the correspondence category of finite $G$-sets. This category has objects all finite $G$-sets, and the morphisms between two finite $G$-sets $S$ and $T$ are the isomorphism classes of finite $G$-sets $U$ over $S\times T$. The composition is given by pullback.

Since these are additive functors, it suffices to determine their values on the orbits $G/H$ as $H$ ranges over conjugacy classes of subgroups. The value of a Mackey functor $\mM$ on $G/H$ has an action of the Weyl group of $H$: $N_{G}(H)/H$, and to every map $G/H\to G/K$, we have a contravariant restriction map
\[
\res_{H}^{K}\colon\mM(G/K)\to\mM(G/H)
\]
and a covariant transfer map
\[
\tr_{H}^{K}\colon\mM(G/H)\to\mM(G/K).
\]
These satisfy the double coset formula, relating the Weyl action to the composite of the restriction and the transfer.

Classical functors like homotopy groups and homology groups have Mackey functor prolongations in the equivariant context. In general, we will denote the Mackey functor version of a functor $F$ by $\m{F}$.

For cyclic $2$-groups, these can easily be depicted graphically a la Lewis: we stack the values $\mM(C_{2^{n}}/C_{2^{k}})$ with $k=n$ at the top, and draw the restriction and transfer maps between the values for $k$ and $k+1$ \cite{Lewis:ROG}. The completely determines the Mackey functor, and makes homological algebra computations quite simple.

\subsubsection{Restriction and Induction}

Given a subgroup $H$ of $G$, there is a natural way to restrict Mackey functors:
\[
\Res_{H}^{G}\mM(T)=\mM(G\times_{H}T).
\]
For cyclic groups, this amounts to only considering the value of $\mM$ on those quotients $C_{2^{n}}/C_{2^{k}}$ for which $C_{2^{k}}\subset H$. We will also simplify notation in this case: if $H=C_{2^{k}}$, then let
\[
\Res_{k}=\Res_{k}^{n}=\Res_{C_{2^{k}}}^{C_{2^{n}}}.
\]
We will also use this notation in the context of the representation ring, where we again have notations of restriction to subgroups.

This functor has both adjoints and they are naturally isomorphic. The left adjoint is called ``induction'', and the right is ``coinduction''. This is given for subgroups $H$ of $G$ by
\[
\Ind_{H}^{G}\mM(T)=\mM(i_{H}^{\ast}T),
\]
where $i_{H}^{\ast}$ is the functor from finite $G$-sets to finite $H$-sets which is simply the obvious restriction. For subgroups of $C_{2^{n}}$, this again looks like the standard induction in representation theory. We will use a similar simplification of notation:
\[
\Ind_{k}=\Ind_{k}^{n}=\Ind_{C_{2^{k}}}^{C_{2^{n}}}.
\]

\subsubsection{Box Product and Green Functors}
Mackey functors have a symmetric monoidal product, the box product, built via Kan extension from the tensor products of the constituent abelian groups. The unit for the box product is the Burnside Mackey functor $\m{A}$ which assigns to an orbit $G/H$ the Grothendieck group of the category of finite $H$-sets. Cartesian product endows this with a product. A monoid for the box product is a Green functor, and another way to describe a Green functor is to say that it is a Mackey functor $\m{R}$ such that
\begin{enumerate}
\item for each $H$, $\m{R}(G/H)$ is a ring and all restriction maps are ring homomorphisms, and
\item if $G/H\to G/K$ is a map of orbits, then the associated transfer map $\m{R}(G/H)\to\m{R}(G/K)$ is a map of $\m{R}(G/K)$-bimodules, where the source inherits a bimodule structure from the restriction homomorphism.
\end{enumerate}

The prototype of a commutative Green functor is a fixed point Mackey functor. 
\begin{definition}
If $M$ is a $G$-module, then the fixed point Mackey functor associated to $M$, $\mM$ has values
\[
\mM(G/H)=M^{H},
\]
restriction maps the canonical inclusions of fixed points, transfer maps the sums over cosets, and Weyl action the obvious residual action on the fixed points.
\end{definition}

If $R$ is a (commutative) ring on which $G$ acts via ring maps, then $\m{R}$ is a (commutative) Green functor.

\begin{definition}
The constant Mackey functor $\mZ$ is the fixed point Mackey functor associated to the trivial action of $G$ on $\mathbb Z$.
\end{definition}

Given any Green functor $\m{R}$, there is an associated category of (left, right, bi-) $\m{R}$-modules. If $\m{R}$ is commutative, then the box product descends to a symmetric monoidal product on the category of $\m{R}$-modules. Moreover, this is a closed symmetric monoidal structure, with internal Hom objects determined by the obvious $Hom$ of $\m{R}$-modules.

In the case of $\mZ$-modules, the literature often uses the older name of ``cohomological Mackey functors'', as the $\mZ$-module structure is determined by the condition the composite of the restriction and transfer is multiplication by the index.

It is convenient to also talk about ``generators'' for a Mackey functor. The cases we will most often consider are actually cyclic.

\begin{notation}
If an $\mR$-module $\mM=\Ind_{H}^{G}\mM'$, then we say that $\mM$ is ``generated by $a\in \mM'(H/H)$'' if the natural map
\[
\Ind_{H}^{G}\Res_{H}^{G}\m{R}\to \mM
\]
adjoint to the map sending $1\in \Res_{H}^{G}\m{R}(H/H)$ to $a\in \mM'(H/H)$ is surjective.

If $\mM$ is generated by $a$, then we will indicate this symbolically by $\mM\{a\}$.
\end{notation}

\subsubsection{Necessary Examples}
We need several families of Mackey functors, all of which will actually be $\mZ$-modules. We will depict them in terms of the Lewisian diagrams described above.

\begin{table}[ht]
\[
\xymatrix@R=.25in@C=.25in{
{\mM} & & {\mZ} & & {\mB_{1,k-1}} & & {\mB_{1,k-1}^{\ast}} & & {\mB_{2,k-2}} \\
{C_{2^{n}}} & & {\Z}\ar@(l,l)[d]^{1} & & {\Z/2}\ar@(l,l)[d]^{1} & & {\Z/2}\ar@(l,l)[d]^{0} & & {\Z/4}\ar@(l,l)[d]^{1} \\
{\vdots} & & {\vdots}\ar@(r,r)[u]^{2}\ar@(l,l)[d]^{1} & & {\vdots}\ar@(r,r)[u]^{0}\ar@(l,l)[d]^{1} & & {\vdots}\ar@(r,r)[u]^{1}\ar@(l,l)[d]^{0} & & {\vdots}\ar@(r,r)[u]^{2}\ar@(l,l)[d]^{1} \\
{C_{2^{k}}} & & {\Z}\ar@(r,r)[u]^{2}\ar@(l,l)[d]^{1} & & {\Z/2}\ar@(r,r)[u]^{0}\ar@(l,l)[d] & & {\Z/2}\ar@(r,r)[u]^{1}\ar@(l,l)[d] & & {\Z/4}\ar@(r,r)[u]^{2}\ar@(l,l)[d]^{1} \\
{C_{2^{k-1}}} & & {\Z}\ar@(r,r)[u]^{2}\ar@(l,l)[d]^{1} & & {0}\ar@(r,r)[u]\ar@(l,l)[d] & & {0}\ar@(r,r)[u]\ar@(l,l)[d] & & {\Z/2}\ar@(r,r)[u]^{2}\ar@(l,l)[d] \\
{C_{2^{k-2}}} & & {\Z}\ar@(r,r)[u]^{2}\ar@(l,l)[d]^{1} & & {0}\ar@(r,r)[u]\ar@(l,l)[d] & & {0}\ar@(r,r)[u]\ar@(l,l)[d] & & {0}\ar@(r,r)[u]\ar@(l,l)[d] \\
{\vdots} & & {\vdots}\ar@(r,r)[u]^{2}\ar@(l,l)[d]^{1}  & & {\vdots}\ar@(r,r)[u]\ar@(l,l)[d] & & {\vdots}\ar@(r,r)[u]\ar@(l,l)[d] & & {\vdots}\ar@(r,r)[u]\ar@(l,l)[d] \\
{C_{1}} & & {\Z}\ar@(r,r)[u]^{2} & & {0}\ar@(r,r)[u] & & {0}\ar@(r,r)[u] & & {0}\ar@(r,r)[u]
}
\]
\caption{Names of Mackey Functors}
\label{tab:Mackey}
\end{table}

The Mackey functors $\mB_{j,k}$ sit in a family of $\mZ$-modules. The $k$ indicates the $2$-adic valuation of the largest subgroup of $C_{2^{n}}$ for which $\Res_{k}\mB_{j,k}=0$. The $j$ indicates the largest possible $2$-adic valuation of the orders of the groups making up Mackey functor. To make this precise, it is helpful to make a small observation about $\mZ$-modules.

\begin{proposition}
If $\mM$ is a $\mZ$-module such that $\mM(C_{2^{n}}/C_{2^{k}})=0$, then $2^{j}$ annihilates $\mM(C_{2^{n}}/C_{2^{k+j}})$.
\end{proposition}
\begin{proof}
The defining feature of $\mZ$-modules is that the composite of the restriction and the transfer is multiplication by the index:
\[
\tr_{k}^{k+j}\res_{k}^{k+j}=2^{j}.
\]
\end{proof}

Thus $\mB_{j,k}$ is the quotient of $\mZ$ by the subMackey functor generated by all classes for subgroups of $C_{2^{k}}$ as well as by $2^{j}$ times any class for any group. The Mackey functor $\mB_{j,k}^{\ast}$ is essentially the same; the roles of restriction and transfer are just swapped.

Later computations will be simplified by a small observation. The following exact sequence is the inclusion of the subMackey functor generated by everything from proper subgroups followed by the quotient map.
\begin{proposition}\label{prop:SESMackey}
We have a short exact sequence of $C_{2^n}$-Mackey functors
\[
0\to \mB_{1,n-2}^{\ast}\to\mB_{2,n-2}\to\mB_{1,n-1}\to 0.
\]
\end{proposition}

\subsection{Signed Induction}
We will need a kind of ``signed'' induction which takes Mackey functors for index $2$ subgroups with trivial Weyl action and promotes them to Mackey functors for $G$ with the sign Weyl action. These occur quite often in the homology of representation spheres when $G$ acts via the sign representation on the underlying homology.

\begin{definition}
If $\mM$ is a Mackey functor for $C_{2^{n-1}}$ with a trivial Weyl action, let $\mM^{-}$ denote the Mackey functor for $C_{2^{n}}$ given by the image of the map:
\[
\Ind_{n-1}^{n}\mM\xrightarrow{1-\gamma}\Ind_{n-1}^{n}\mM,
\]
where the map labeled $1-\gamma$ is adjoint to the anti-diagonal map 
\[
\mM\xrightarrow{\begin{bmatrix}
1\\
-1
 \end{bmatrix}}\mM\oplus\mM\cong\Res_{n-1}^n\Ind_{n-1}^n\mM.
\]
\end{definition}

\subsection{Pullbacks}

We need a final operation on Mackey functors. Let $N$ be a normal subgroup of $G$, and let $Q=G/N$. Then we have a pullback or ``inflation'' functor from $Q$-Mackey functors to $G$-Mackey functors.

\begin{definition}
Let $\pi_N^\ast$ be the functor from $Q$-Mackey functors to $G$-Mackey functors defined by
\[
\pi_N^\ast(\mM)(T)=\mM(T^N).
\]
All structure maps are determined by first taking $N$-fixed points (in the Burnside category) and then applying $\mM$.
\end{definition}

The functor $\pi_N^\ast$ is a fully faithful embedding of $Q$-Mackey functors into $G$-Mackey functors and has both adjoints. 
%
%
Since $\pi_N^\ast$ has both adjoints, it commutes with all adjoints. 
\begin{corollary}
If $N\subset H\subset G$ and $N$ is normal in $G$, then we have a natural isomorphism of functors
\[
\pi_N^\ast\Ind_{H/N}^{G/N}(-)\cong \Ind_H^G \pi_N^\ast(-),
\]
where the first $\pi_N^\ast$ is the functor from $G/N$-Mackey functors to $G$-Mackey functors, while the latter is the functor from $H/N$-Mackey functors to $H$-Mackey functors.
\end{corollary}

\subsection{Representations}

We also need several representations of cyclic $2$ groups.

\begin{definition}\mbox{}
\begin{itemize}
\item Let $\sigma_{k}$ denote the sign representation of $C_{2^{k}}$.
\item For $k\geq 2$, let $\lambda_{k}(m)$ denote the representation induced by the composite $C_{2^{k}}\to S^{1}\xrightarrow{z\mapsto z^{m}} S^{1}$.
\item For $k\geq 2$, let $\lambda_{k}'=\lambda_{k}(2^{k-2})$.
\item Let $\rho_{k}$ denote the regular representation of $C_{2^{k}}$.
\item Let $\bar{\rho}_{k}$ denote the reduced regular representation for $C_{2^{k}}$, the quotient of $\rho_{k}$ by the trivial representation.
\end{itemize}
\end{definition}

When there is no ambiguity, we will suppress the subscripts to avoid notational clutter. Since we care not about the representations themselves but rather the $2$-local $JO$-equivalence classes, we abusively say that two representations are equivalent of $S^{V}\simeq S^{W}$. In particular, if the $2$-adic valuations of $m$ is $\ell$, then $\lambda_{n}(m)$ and $\lambda_{n}(2^{\ell})$ are equivalent.

\section{\texorpdfstring{$RO(G)$}{RO(G)}-graded Homotopy and Homology of Representation Spheres}

\subsection{Euler Classes}

\begin{definition}
If $V$ is a representation of a finite group $G$ such that $V^{G}=\{0\}$, then let $a_{V}\colon S^{0}\to S^{V}$ denote the Euler class of $V$, viewed as a bundle over a point.
\end{definition}

The follow facts are immediate from the definition.

\begin{proposition}
\mbox{}
\begin{enumerate}
\item The map $a_{V}$ is essential, while for any $H$ such that $V^{H}\neq \{0\}$, the restriction $\res_{H}^{G}a_{V}\simeq 0$.
\item The maps $a_{V\oplus W}$ and $a_{V}\cdot a_{W}$ are equal.
\item The element $a_{V}$ generates $\m{\pi}_{0}S^{V}$.
\item For $C_{2^{n}}$, if $k$ is the largest integer such that $V^{C_{2^{k}}}\neq\{0\}$, then $2^{n-k}$ annihilates the Hurewicz image of $a_{V}$.
\end{enumerate}
\end{proposition}
We will also let $a_{V}$ denote the Hurewicz image of $a_{V}$ in $\pi_{\star}H\mZ$.

There is a close connection between restrictions, transfers, and multiplication by Euler classes of sign representations. This is undoubtedly known, but we include it here for completeness.

\begin{lemma}\label{lem:RestrictionsExtensions}
Let $H$ be an index two subgroup of $G$, and let $\sigma=\sigma_{G/H}$ be the sign representation of the quotient group.
\begin{enumerate}
\item If $a_{\sigma}x=0$, then there is a class $y$ such that $\tr_{H}^{G}(y)=x$.

\item If $x$ is in the kernel of multiplication by $a_{\sigma}$, then $\res_{H}^{G}x\neq 0$.
\end{enumerate}
\end{lemma}
\begin{proof}
Identify $C_{2}$ with $G/H$. We have a cofiber sequence
\[
C_{2+}\to S^{0}\xrightarrow{a_{\sigma}} S^{\sigma},
\]
and this will prove both parts.

Smashing a spectrum $X$ with this cofiber sequence and taking $\mpi_{\ast}$ gives a long exact sequence
\[
\dots\to\mpi_{k}(C_{2+}\wedge X)\xrightarrow{\m{tr}_{H}^{G}}\mpi_{k}(X)\xrightarrow{a_{\sigma}} \mpi_{k}(S^{\sigma}\wedge X)\to \mpi_{k-1}(C_{2+}\wedge X)\to\dots,
\]
and the Wirthm\"{u}ller isomorphism gives an identification 
\[
\mpi_{k}(C_{2+}\wedge X)\cong\Ind_{H}^{G}\Res_{H}^{G}\mpi_{k}(X).
\]
The first map shown above the Mackey functor map arising from $\tr_{H}^{G}$, which identifies the kernel of $a_{\sigma}$ multiplication with the image of the transfer from $H$.

If we instead take maps out of the cofiber sequence, and use the obvious identification of the homotopy groups of $F(C_{2+},X)$, we get a long exact sequence
\[
\dots\leftarrow\Ind_{H}^{G}\Res_{H}^{G}\mpi_{k}(X)\xleftarrow{\m{res}_{H}^{G}}\mpi_{k}(X)\xleftarrow{a_{\sigma}} \mpi_{k}(S^{-\sigma}\wedge X)\dots.
\]
The first named map is the Mackey functor map arising from $\res_{H}^{G}$, which then identifies the cokernel of multiplication by $a_{\sigma}$ with those maps supporting a non-trivial restriction.
\end{proof}

\begin{remark}
In \cite{HHRC4}, a stronger version of this argument is used, describing the above result in terms of a spectral sequence of Mackey functors and maps thereof. We will not need this more general form, so we ignore it here.
\end{remark}

\subsection{Orientation Classes}
If $V$ is an orientable representation of $C_{2^{n}}$ (meaning it is defined by a map $C_{2^{n}}\to SO(V)$, rather than to $O(V)$), then 
\[
\mH_{\dim V}(S^{V};\mZ)\cong \mZ.
\]
A choice of generator is an orientation of $S^{V}$ for $H\mZ$. Fixing once and for all compatible orientations of $S^{n}$ for all $n$, we can compatibly choose generators $u_{V}$ for all orientations $V$. These have a similar list of properties, the second of which is a special case of an obvious more general relation which carves out essentially all relations in the $RO(G)$-graded homotopy. The form we give is the only one we need, however.
\begin{proposition}
\mbox{}
\begin{enumerate}
\item The classes $u_{V\oplus W}$ and $u_{V}\cdot u_{W}$ are equal.
\item We have an equality $a_{2\sigma}u_{\lambda'}=2a_{\lambda'}u_{2\sigma}$.
\end{enumerate}
\end{proposition}

We also need a small observation about representations like $\rho_{k}$ which are not orientable. The restriction of $\rho_{k}$ to $C_{2^{k-1}}$ is $2\rho_{k-1}$. This representation is orientable, and so has homology in every even degree between $2$ and $2^{k}$. There is no Bredon homology for $\rho_{k}$ in these degrees, however, due to the failure of orientability.

\begin{definition}
If $V$ is a non-orientable representation of $V$ and if $x$ generates $\mH_{2k}(S^{\res_{n-1}V};\mZ)$, then let $u^{-}x$ denote the generator of
\[
\mH_{2k}(S^{V};\mZ)\cong \big(\mH_{2k}(S^{\res_{n-1}V};\mZ)\big)^{-}.
\]
\end{definition}

\subsection{Bredon Homology of Representation Spheres}
For cyclic $p$-groups, it is easy to compute the homology of any representation sphere with coefficients in any Mackey functor (see, for instance, \cite{Hill:2009uo, CDMII, HillICM}). We need only the homology with coefficients in $\mZ$. We present here a simplified description, the details of which can be found in \cite{HillICM}.

\begin{definition}
If $W$ is an orthogonal representation of $C_{2^{n}}$ with $W^{C_{2^n}}=\{0\}$, then let $\mathcal V_{W}$ denote set of $JO$-equivalence classes of orthogonal decompositions
\[
W=V'\oplus V''
\]
such that
\begin{enumerate}
\item $V''$ is orientable
\item if $H\subset C_{2^{n}}$ stabilizes a non-zero vector of $V''$, then $H$ stabilizes all of $V'$.
\end{enumerate}

If $V'\oplus V''\in\mathcal V_{W}$, let 
\[
m(V')=\max \big\{m | (V')^{C_{2^{m}}}\neq\{0\}\big\}
\]
\end{definition}

Since $JO$-equivalence of an irreducible representation is determined by stabilizer orders, and since orientability is determined by the parity of the number of sign summands, there are very few such decompositions.
\begin{proposition}
If $W^{C_{2^n}}=\{0\}$, then for each even integer $k$ between $0$ and $\dim W$, there is a unique $V'\oplus V''\in \mathcal V_{W}$ such that $\dim V''=k$.
\end{proposition}

\begin{example}
If $G=C_{8}$ and $W=3\sigma+\lambda(2)+\lambda$, then we have $\mathcal V_W$ consists for $4$ elements:
\begin{center}
\begin{tabular}{|c|c|c|}
\hline
$\dim V''$ & $V'$ & $V''$ \\ \hline
$6$ & $\sigma$ & $2\sigma+\lambda(2)+\lambda$ \\ \hline
$4$ & $3\sigma$ & $\lambda(2)+\lambda$ \\ \hline
$2$ & $3\sigma+\lambda(2)$ & $\lambda$ \\ \hline
$0$ & $3\sigma+\lambda(2)+\lambda$ & $0$ \\ \hline
\end{tabular}
\end{center}
\end{example}

\begin{theorem}
If $W$ is orientable and $W^{C_{2^{n}}}=\{0\}$, then for $V'\oplus V''\in\mathcal V_{W}$
\[
\mH_{\dim V''}(S^{W};\mZ)=\mB_{n-m(V'),m(V')} \cdot a_{V'}u_{V''},
\]
while all other homology groups are zero.
\end{theorem}

\begin{theorem}
If $W$ is not orientable with $W^{C_{2^{n}}}=\{0\}$, then $W=\sigma\oplus W'$ with $W'$ orientable and fixed point free. If $V'\oplus V''\in \mathcal V_{W'}$, then
\[
\mH_{\dim V''}(S^{W};\mZ)=\mB_{1,n-1}\cdot a_{\sigma}a_{V'}u_{V''}.
\]
All other even groups vanish. The odd groups satisfy
\[
\mH_{2k+1}(S^{W};\mZ)\cong \mH_{2k+1}(S^{\res_{n-1}W+1};\mZ)^{-}\cong \mH_{2k}(S^{\res_{n-1}W};\mZ)^{-}.
\]
\end{theorem}

\section{The Slice Filtration for \texorpdfstring{$\Xi_{n}$}{Xin}}\label{sec:SliceRecollement}
Essentially all of the material in this section is either in \cite{Hill:2009uo} or follows easily from the techniques described there and in \cite{CDMII}.

\subsection{The slice tower for \texorpdfstring{$\Xi_{n}$}{Xin}}

Our computation of the homotopy Mackey functor $\mpi_{3}(\Xi_{n})$ relies on an equivariant filtration, the slice filtration, developed in \cite{Hill:2009uo}. We recall the salient properties here. As a minor piece of notation, let $\gamma$ denote a chosen generator for $C_{2^{n}}$. If ${r}$ is a homotopy class, let 
\[
C_{2^{n}}\cdot {r}=\{{r},\gamma{r},\dots,\gamma^{2^{n-1}-1}{r}\}.
\]

\begin{theorem}
There are classes 
\[
r_{i}\colon S^{2i}\to i_{e}^{\ast}\Xi_{n}
\]
such that we have an isomorphism
\[
\pi_{\ast}i_{e}^{\ast}\Xi_{n}=\Z[C_{2^{n}}\cdot{r}_{1}^{C_{2^{n}}}, C_{2^{n}}\cdot{r}_{2},\dots],
\]
where
\[
\gamma\cdot \gamma^{j}r_{i}=\begin{cases}
\gamma^{j+1}r_{i} & j+1<2^{n-1} \\
(-1)^{i}r_{i} & j+1=2^{n-1}.
\end{cases}
\]
\end{theorem}

Though not explicit in the notation, the classes $r_{i}$ depend heavily on $n$. If we need to distinguish between the $r_{i}$ for different $n$, we will use the corresponding groups as superscripts.

One of the main technical results in \cite{Hill:2009uo} is the determination of the slice tower of $\Xi_{n}$. For this, we use a slight refinement of the notation therein, simplifying several formulae.

\begin{definition}
Let 
\[
\mathcal P_{n} = \{\text{monomials in } \F_{2}[C_{2^{n}}\cdot {r}_{1}, C_{2^{n}}\cdot r_{2},\dots]\},
\]
where $|r_{i}|=2i$, and where $\gamma^{2^{n}}r_{i}=r_{i}$.

If $p\in P_{n}$, then let
\begin{enumerate}
\item $|p|$ be the underlying degree of $p$, 
\item $H(p)$ be the stabilizer of $p$ in $C_{2^{n}}$, and
\item $||p||=\tfrac{|p|}{|H(p)|}\rho_{H(p)}$.
\end{enumerate}
\end{definition}

\begin{theorem}[{\cite[Proposition~5.32]{Hill:2009uo}}]\label{thm:BarrisArePermCycles}
If $p\in\mathcal P_{n}$, then there is an $H(p)$-equivariant lift
\[
\bar{p}\colon S^{||p||}\to i_{H(p)}^{\ast}\Xi_{n}
\]
of the non-equivariant map $p\in\pi_{|p|}i_{e}^{\ast}\Xi_{n}$.
\end{theorem}

\begin{theorem}[{\cite[Theorem 6.1]{Hill:2009uo}}]\label{thm:SliceTheorem}
There is a multiplicative filtration of $\Xi_{n}$, the slice filtration, with associated graded given by
\[
\bigvee_{p\in \mathcal P_{n}/C_{2^{n}}} C_{2^{n}+}\wedge_{H(p)} S^{||p||}\wedge H\mZ.
\]
\end{theorem}

Theorem~\ref{thm:BarrisArePermCycles} guarantees that for each $p\in\mathcal P_{n}/C_{2^{n}}$, the natural Hurewicz map
\[
C_{2^{n}+}\wedge_{H(p)} S^{||p||}\to C_{2^{n}+}\wedge_{H(p)} S^{||p||}\wedge H\mZ
\]
detects a non-trivial element in $\pi_{\star}\Xi_{n}$ and hence is a permanent cycle.

Since by definition
\[
\mpi_{\ast}\big(C_{2^n+}\wedge_{C_{2^k}} S^{m\rho_{k}}\wedge H\mZ\big)\cong
\Ind_k \m{H}_{\ast}\big(S^{m\rho_k};\mZ\big),
\]
we have a very easy to compute $E_2$-term for the spectral sequence associated to the slice filtration. This requires only knowledge of the Bredon homology of various representation spheres. Since we care only about what happens in very low topological degree, we first determine exactly what can contribute to each degree and  filtration.

The group $C_{2^{n}}$ acts by permuting the coordinates of $\mathbb F_{2}[C_{2^{n}}\cdot r_{1},\dots]$ and it therefore acts on $\mathcal P_{n}$. Theorem~\ref{thm:SliceTheorem} shows that to understand the contribution of an orbit of a monomial, we need to understand the associated stabilizer subgroup and the relative degrees. We first determine all of the fixed points.

\begin{definition}
For $1\leq k\leq n$, let
\[
N_{1}^{k}r_{i}=\prod_{j=0}^{2^{k-1}-1} \gamma^{2^{n-k}j} r_{i}.
\]
\end{definition}
The ``$N$'' in these classes names reflects the fact that they are classes underlying the norm functor applied to the classes $r_{i}$. In fact, if $p=N_{1}^{k}r_{i}$, then the class $\bar{p}$ is adjoint to the internal norm:
\[
N_{1}^{k}\bar{r_{i}}\colon S^{i\rho_{k}}\to i_{C_{2^{k}}}^{\ast}\Xi_{n}.
\]

Because of the group $H=C_{2^{k}}$ acts via multiplication by the powers of $\gamma^{2^{n-k}}$, the following proposition is obvious.
\begin{proposition}
If $H=C_{2^{k}}$, then the $C_{2^{n}}$-set $\mathcal P_{n}^{H}$ is the collection of monomials in
\[
\F_{2}[N_{1}^{k}r_{1},\gamma N_{1}^{k}r_{1},\dots, \gamma^{2^{n-k}-1}N_{1}^{k}r_{1},N_{1}^{k}r_{2},\dots].
\]
\end{proposition}

We now regrade $\mathcal P_{n}$, grading instead by subgroups of $C_{2^{n}}$ and then degrees of monomials in the norm classes.

\begin{definition}
If $C_{2}\subset H\subsetneq C_{2^{n}}$, then let $K$ be the subgroup of $C_{2^{n}}$ in which $H$ has index $2$. Let
\[
\mathcal P_{n,m}(H)=\big(\mathcal P_{n}^{H}\big)_{m|H|}-\big(\mathcal P_{n}^{K}\big)_{m|H|},
\] 
where $(\mathcal P_{n}^{H})_{m|H|}$ refers to the elements of degree $m|H|$ in the $H$-fixed points of $\mathcal P_{n}$ and similarly for $K$, and let 
\[
\mathcal P_{n,m}(C_{2^{n}})=\big(\mathcal P_{n}^{C_{2^{n}}}\big)_{2^{n}m}.
\] 
\end{definition}

The $\mathcal P_{n,m}(H)$ as $m$ and $H$ vary form a partition of $\mathcal P_{n}$ that is especially amenable to computation. Our determination of the Bredon homology of representation spheres above shows the following, since $\mathcal P_{n,m}(H)$ gives all wedge summands of the form $S^{m\rho_{H}}$.

\begin{proposition}\label{prop:BoundsonPn}
Let $H=C_{2^{k}}$. Then elements in $\mathcal P_{n,2m}(H)/C_{2^{n}}$ contribute to the slice $\m{E}_{2}$ term in degrees 
\[
(t-s,s)=(2j,2^{k+1}m)
\]
for $j$ between $m$ and $2^{k}m$, inclusive.

The elements in $\mathcal P_{n,2m+1}(H)/C_{2^{n}}$ contribute to the slice $\m{E}_{2}$ term in degrees 
\[
(t-s,s)=(2j+1,2^{k}(2m+1)-2j-1)
\]
for $j$ between $m$ and  $2^{k}m+2^{k-1}-1$, inclusive, and in also in degrees
\[
(t-s,s)=(2j,2^{k}(2m+1)-2j)
\]
for $j$ between $2m+1$ and $2^{k-1}(2m+1)$ inclusive.
\end{proposition}

%
%
%

\subsection{A slice differential}

In \cite{Hill:2009uo}, we needed only determine a single differential. This was sufficient to show that the classes we needed were appropriately periodic, giving the $256$-fold periodicity observed there. To describe it, we use a small piece of notation which will also simplify later computation.
\begin{definition}
Let 
\[
\fancydelta_{k}=N_{1}^{n}\bar{r}_{2^k-1}.
\]
\end{definition}

As before, the class $\fancydelta$ depends on $n$, though it is omitted from the notation. If we need to distinguish, we will again use a superscript.

\begin{theorem}\label{thm:Hill:2009uoDifferentials}
For each $n$, we have differentials
\[
d_{2^{k+1}+(2^{k+1}-1)(2^{n}-1)}(u_{2\sigma_{n}}^{2^{k}})=a_{\sigma}^{2^{k+1}}a_{\bar{\rho}_{n}}^{2^{k+1}-1}\fancydelta_{k+1}.
\]
\end{theorem}

The argument given therein actually shows more.
\begin{proposition}
For each $n$ and for each $k$, the class $2u_{2\sigma_{n}}^{2^{k}}$ is a permanent cycle.
\end{proposition}
\begin{proof}
This class is the transfer of the map identifying $i_{C_{2^{n-1}}}^\ast S^{2^{k+1}\sigma}$ with  $S^{2^{k+1}}$, making it a permanent cycle. 
\end{proof}

We will need in what follows a way to compare the classes $\fancydelta_{1}^{H}$ and $\fancydelta_{1}^{K}$ for subgroups $H$ and $K$ of $C_{2^{n}}$. It is easiest to compare all of these in terms of the polynomial generators $\bar{r}_{1}^{C_{2^{n}}},\dots$. 

In general, the classes $\bar{r}_{k}^{C_{2^n}}$ depend heavily on $n$, and the formulas describing $\bar{r}_{k}^{C_{2^{n-1}}}$ in terms of the generators $\bar{r}_{\leq k}^{C_{2^{n}}}$ are quite complicated and depend heavily on the underlying bordism problem presented by the restriction map for $\Xi_{n}$. However, for $\bar{r}_{1}^{C_{2^{n}}}$, there is a very simple formula.

\begin{proposition}
The classes $\bar{r}_{1}^{C_{2^{n-k}}}$ are given in terms of the classes $\bar{r}_{1}^{C_{2^{n}}},\dots$ by the formula:
\[
\bar{r}_{1}^{C_{2^{n-k}}}=\sum_{j=0}^{2^{k}-1}\gamma^{j}\bar{r}_{1}^{C_{2^{n}}}.
\]
\end{proposition}
\begin{proof}
It obviously suffices to show this for $k=1$. Here it follows immediately from the definition of the classes $\bar{r}_{k}^{C_{2^n}}$ in terms of coefficients of the universal isomorphism linking the formal group on $\Xi_{n}$ with its $\gamma$ conjugate \cite[Corollary~5.49]{Hill:2009uo}.
\end{proof}

Since the norm functor is additive modulo the transfer and since it commutes with the Weyl action, this immediately gives us another formula linking the norms of these classes.
\begin{corollary}
Modulo transfers coming from smaller subgroups, we have
\[
N_{1}^{n-k}\bar{r}_{1}^{C_{2^{n-k}}}\equiv \sum_{j=0}^{2^{k}-1}\gamma^{j}N_{1}^{n-k}\bar{r}_{1}^{C_{2^{n}}}.
\]
\end{corollary}
In particular, since $a_{\bar{\rho}_{n-k}}$ annihilates all transfers from proper subgroups of $C_{2^{n-k}}$, multiplication by it gives an exact formula for the norms of $\bar{r}_{1}^{C_{2^{n-k}}}$.
\begin{corollary}
For any $k$, we have
\[
a_{\bar{\rho}_{n-k}}\fancydelta_{1}^{C_{2^{n-k}}}=
a_{\bar{\rho}_{n-k}}N_{1}^{n-k}\bar{r}_{1}^{C_{2^{n-k}}}=a_{\bar{\rho}_{n-k}}\sum_{j=0}^{2^{k}-1}\gamma^{j}N_{1}^{n-k}\bar{r}_{1}^{C_{2^{n}}}.
\]
\end{corollary}
In particular, the last term is the sum over the Weyl group of the element $N_{1}^{n-k}\bar{r}_{1}^{C_{2^{n}}}$, a common expression for the transfer from the subgroup $C_{2^{n-k}}$ in an induced Mackey functor.

\section{The Slice Spectral Sequence for \texorpdfstring{$\m{\pi}_{3}\Xi$}{pi3}}\label{sec:SliceSS}
\subsection{A new slice differential and extension}\label{sec:NewDifferential}
To analyze the $3$-stem, we have to analyze the start of the analogous pattern for the $u_{\lambda'}$ family.

The strategy of proof starts very similarly: smash the slice tower for $\Xi_{n}$ with a representation sphere and determine the slice differentials. This is greatly simplified by the huge stabilizer of $\sigma_{n}$ and of $\lambda'$, where $\lambda'$ is the rotation by $\tfrac{\pi}{2}$ representation.

\begin{prop}
The only slices which contribute to $\m{\pi}_{1}$ and $\m{\pi}_{2}$ of $S^{\lambda'+\epsilon\sigma}\wedge \Xi_{n}$ are:
\begin{multline*}
\big(S^{0}\wedge H\mZ\big)\vee \big((C_{2^{n}+}\wedge_{C_{2^{n-1}}}S^{\rho_{n-1}})\wedge H\mZ\big)\\ \vee \big((C_{2^{n}+}\wedge_{C_{2^{n-1}}}S^{2\rho_{n-1}}\vee S^{\rho_{n}})\wedge H\mZ\big)\vee \big(S^{2\rho_{n}}\wedge H\mZ\big).
\end{multline*}
\end{prop} 

\begin{proof}
Since $\Res_{n-1}^{n}\lambda'=2\sigma_{n-1}$ and since $\Res_{n-1}^{n}\sigma_{n}=1$, we have
\[
S^{\lambda'+\epsilon\sigma}\wedge \big(C_{2^{n}+}\wedge_{C_{2^{k}}} S^{m\rho_{k}}\big)=\begin{cases}
C_{2^{n}+}\wedge_{C_{2^{k}}} S^{2+\epsilon+m\rho_{k}} & k\leq n-2 \\
C_{2^{n}+}\wedge_{C_{2^{k}}} S^{\epsilon+2\sigma+m\rho_{k}} & k=n-1 \\
S^{\epsilon\sigma+\lambda'+m\rho_{n}} & k=n.
\end{cases}
\]

In particular, since $m\geq 1$ for all of the induced slices for $\Xi_{n}$, we see that none of the slices induced from $C_{2^{k}}$ for $k\leq n-2$ contribute to the slice $\m{E}_{2}$ term for $t-s=1, 2$. Thus the slice $\m{E}_{2}$ term is the same as if we had only slices induced from $C_{2^{n-1}}$ and those which are not induced. Similarly, nothing for $m\geq 3$ can contribute to $t-s=1,2$. 
\end{proof}

With this simplification, it is easy to analyze the fate of $u_{\lambda'}$ and of $a_{\sigma_{n}}u_{\lambda'}$.

\begin{theorem}\label{thm:Differential}
The class $u_{\lambda'}$ supports a $d_{2^{n-1}+1}$-differential:
\[
d_{2^{n-1}+1}(u_{\lambda'})=a_{\lambda'} \tr_{n-1}^{n} (a_{\bar{\rho}_{n-1}} N_1^{n-1}\bar{r}_1).
\]
\end{theorem}

\begin{proof}
In this situation, the slice $\m{E}_{2}$-term is very simple. We show this for $n=3$ as Figure~\ref{fig:SliceE2ulambda}; for $n$ at least $2$, all of the pictures are the same; they are simply scaled. In this picture, a bullet indicates $\mB_{1,k}$ for the $C_{2^{k+1}}$-Mackey functor, a box indicates the constant Mackey functor $\mZ$, while a circle indicates $\mB_{2,k-1}$. A subscript that is a minus indicates the Mackey functor that has the same form as the one without the minus, but with the sign Weyl action.

\begin{figure}[ht]
\begin{minipage}[b]{.45\textwidth}
\centering
\includegraphics{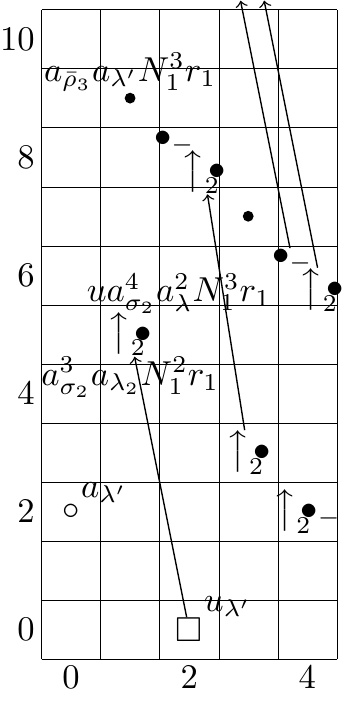}
\caption{The slice spectral sequence for $S^{\lambda'}\wedge\Xi_{3}$}
\label{fig:SliceE2ulambda}
\end{minipage}
\begin{minipage}[b]{.45\textwidth}
\centering
\includegraphics{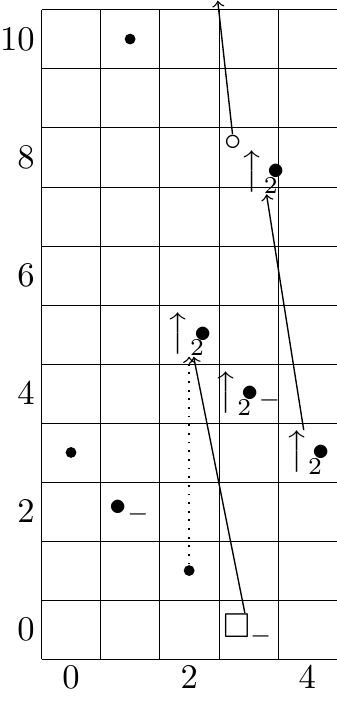}
\caption{The slice spectral sequence for $S^{\sigma+\lambda'}\wedge\Xi_{3}$}
\label{fig:PermCycle} 
\end{minipage}
\end{figure}

The $\mZ$ in $\m{E}_{2}^{2,0}$ is generated by $u_{\lambda'}$ in $G/G$. This restricts to $u_{2\sigma_{n-1}}$, which supports a $d_{2^{n-1}+1}$-differential. Since this is a spectral sequence of Mackey functors, we have the following commutative diagram:
\[
\xymatrix{
{u_{\lambda'}}\ar[d]^{\res_{n-1}^{n}} \ar[rr]^{d_{k\leq 2^{n-1}+1}}& & {?}\ar[d]^{\res_{n-1}^{n}} \\
{u_{2\sigma_{n-1}}}\ar[rr]_-{d_{2^{n-1}+1}} & & {a_{\sigma_{n-1}}^{2}a_{\bar{\rho}_{n-1}}\fancydelta_{1}^{C_{2^{n-1}}}}
}
\]
and since the first possible differential on $u_{\lambda'}$ is, for degree reasons,
\[
d_{2^{n-1}+1}(u_{\lambda'})=a_{\lambda'}\tr_{n-1}^{n} a_{\bar{\rho}_{n-1}}N_{1}^{n-1}\bar{r}_{1},
\]
we conclude that this is the differential on $u_{\lambda'}$.
\end{proof}

\begin{theorem}\label{thm:Cycle}
The class $a_{\sigma}u_{\lambda'}$ is a permanent cycle.
\end{theorem}
\begin{proof}
Since the image of the differential on $u_{\lambda'}$ is in the image of the transfer, it is killed by $a_{\sigma_{n}}$-multiplication. Thus $a_{\sigma}u_{\lambda'}$ survives to $\m{E}_{2^{n-1}+2}$. However, by the same analysis as before, there is only one possible target for a differential on $a_{\sigma}u_{\lambda'}$: the class $a_{\sigma}a_{\lambda'}a_{\bar{\rho}_{n}}\fancydelta_{1}$. The example case is again $C_{8}$, which is depicted in Figure~\ref{fig:PermCycle}.

Since on $\m{E}_{2^{n-1}+2}$ the target supports arbitrary $a_{\sigma}$-multiplication while the source is annihilated by $a_{\sigma}^{2}$, we conclude that $a_{\sigma}u_{\lambda'}$ is a permanent cycle.
\end{proof}

\begin{corollary}\label{cor:Extension}
The element $a_{\sigma}u_{\lambda'}$ supports a non-trivial restriction to 
\[
a_{\lambda'} \tr_{n-1}^{n} (a_{\bar{\rho}_{n-1}}N_1^{n-1}\bar{r}_1)\in 
H_1\big(S^{\lambda'}\wedge(C_{2^{n}+}\wedge_{C_{2^{n-1}}}S^{\rho_{n-1}})\big)\cong \m{E}_2^{2^{n-1}+1,2^{n-1}},
\] 
jumping $2^{n-1}$ filtrations.
\end{corollary}
\begin{proof}
Since $u_{\lambda'}$ supports a differential, the class $a_{\sigma}u_{\lambda'}$ is not in the image of $a_{\sigma}$-multiplication. Lemma~\ref{lem:RestrictionsExtensions} shows that it therefore supports a non-trivial restriction. There is only one possible target for the extension, since all other classes are wiped out by differentials.
\end{proof}

\begin{corollary}
A class in $\m{\pi}_{2-\sigma-\lambda}\Xi_{n}(C_{2^n}/C_{2^n})$ represented by the element $a_{\sigma_{n}}u_{\lambda'}$ is at least $4$-torsion.
\end{corollary}

\subsection{The slice spectral sequence $E_2$ term for $t-s<5$}
The main results from \S~\ref{sec:SliceRecollement} show that the only slices which can contribute to $\m{\pi}_{3}\Xi$ are those of the form
\[
C_{2^{n}+}\wedge_{C_{2^{k}}}S^{m\rho_{k}}\wedge H\mZ
\]
for $m\leq 3$. Additionally, the $C_{2^n}$-analogue of Quillen's splitting of $MU_{(2)}$ into a wedge of copies of $BP$ shows that we need only consider those $\bar{r}_{k}$ with $k=2^j-1$ \cite{Hill:2009uo}. Thus we need only consider those summands which correspond to monomials in norms of $r_{1}$ and of $r_{3}$. 

The bounds given by Proposition~\ref{prop:BoundsonPn} allow us to quickly see which $\mathcal P_{n,m}(H)$ can contribute in each degree.

For $m=1$, there is a unique orbit of monomials: 
\[
\mathcal P_{n,1}(C_{2^{k}})=\{
N_{1}^{k}r_{1},\dots,\gamma^{2^{n-k-1}-1}N_{1}^{k}r_{1}
\}.
\]
This gives immediately the following proposition.
\begin{proposition}
The slice $\m{E}_{2}$-term in $t-s=1$ is given by
\[
\m{E}_{2}^{s,t}=\begin{cases}
\Ind_{k}\mB_{1,k-1}\cdot\{a_{\bar{\rho}_{k}}N_{1}^{k}\bar{r}_{1}^{C_{2^{n}}}\} & s=2^{k}-1, t=2^{k}, k\leq n \\
0 & \text{ otherwise.}
\end{cases}
\]
\end{proposition}
\begin{proof}
Only slices of the form $C_{2^{n}+}\wedge_{C_{2^{k}}}S^{\rho_{k}}\wedge H\mZ$ can contribute to $t-s=1$, and the result is then immediate.
\end{proof}

From this point on, the $C_{2^{n}}$-sets $\mathcal P_{n,m}(C_{2^{k}})$ are not transitive (and in fact, grow exponentially!). 

\begin{proposition}
The slice $\m{E}_{2}$-term for $t-s=2$ is given by
\[
\m{E}_{2}^{2^{k+1}-2,2^{k+1}}=
\bigoplus_{p\in P_{n,2}(C_{2^{k}})/C_{2^{n}}}\Ind_{k}\mB_{1,k-1}\cdot\{a_{\bar{\rho}_{k}}^{2}\bar{p}\}\oplus\Ind_{k+1}\mB_{1,k-1}^{-}\cdot\{u^{-}a_{\bar{\rho}_{k-1}}^{2}\res_{k-1}N_1^k r_{1}\},
\]
for $0<k\leq n$. 
We have
\[
\m{E}_{2}^{0,2}=
\Ind_{1}\mZ^{-}\cdot\{u^{-} r_{1}\}.
\]
All other Mackey functors are zero.
\end{proposition}
\begin{proof}
Only $\mathcal P_{n,2}(C_{2^{k}})$ and $\mathcal P_{n,1}(C_{2^{k}})$ can contribute to $\mpi_{2}$, and this is essentially the direct sum decompositions above.
The first summand $\mpi_{2}$ of the slices indexed by $\mathcal P_{n,2}(C_{2^{k}})/C_{2^{n}}$. These are all of the slices of the form
\[
C_{2^{n}+}\wedge_{C_{2^{k}}} S^{2\rho_{k}}\wedge H\mZ.
\]

The second summand arises from $\mpi_{2}$ of the slices indexed by $\mathcal P_{n,1}(C_{2^{k+1}})/C_{2^{n}}$ for $k\geq 1$:
\[
C_{2^{n}+}\wedge_{C_{2^{k+1}}} S^{\rho_{k+1}}\wedge H\mZ.
\]

Finally, we have the term indexed by $\mathcal P_{n,1}(C_{2})$. The restriction of $S^{\rho_{1}}$ to the index two subgroup of $C_{2}$ is just $S^{2}$; that gives the copy of $\mZ^{-}$ seen.
\end{proof}

\begin{proposition}
The slice $\m{E}_{2}$-term for $t-s=3$ is given by
\[
\m{E}_{2}^{2^{k}-3,2^{k}}=\Ind_{k}\mB_{1,k-1}\cdot\{a_{\sigma}u_{\lambda'}a_{\bar{\rho}-\sigma-\lambda'} N_{1}^{k}r_{1}\},
\]
for $k\geq 2$, and for all $k\geq 1$
\[
\m{E}_{2}^{3\cdot 2^{k}-3,3\cdot 2^{k}}=
\bigoplus_{\bar{p}\in P_{n,3}(C_{2^{k}})/C_{2^{n}}}\Ind_{k}\mB_{1,k-1}\cdot \{a_{\bar{\rho}_{k}}^{3}\bar{p}\}.
\]
For all other filtrations, the Mackey functors are zero.
\end{proposition}
\begin{proof}
Only $\mathcal P_{n,1}(C_{2^{k}})$ for $k\geq 2$ and $\mathcal P_{n,3}(C_{2^{k}})$ for all $k$ can contribute to $t-s=3$ by Proposition~\ref{prop:BoundsonPn}. 

For $\mathcal P_{n,1}(C_{2^{k}})$, the contribution to $\mpi_{3}$ is the Mackey functors for $s=2^{k}-3$ above. The indices given by $\mathcal P_{n,3}(C_{2^{k}})$ contribute in Mackey functors listed filtration $s=3\cdot 2^{k}-3$.
\end{proof}

By Theorem~\ref{thm:Differential}, the classes $a_{\sigma}u_{\lambda'}$ are permanent cycles. The classes $a_{\bar{\rho}_{k}}$ and $a_{\bar{\rho}_{k}-\sigma-\lambda'}$ are permanent cycles, as they come from the sphere, and the polynomials $\bar{p}$ and $N_{1}^{k}r_{i}$ are also, as we have explicit representatives for them in the homotopy given by Theorem~\ref{thm:BarrisArePermCycles}.

\begin{proposition}
Every element in the slice $\m{E}_{2}$ term for $t-s=3$ is a permanent cycle.
\end{proposition}

We need only determine enough of the putative $4$ stem to determine the boundaries.

\begin{proposition}
The slice $\m{E}_{2}$-term for $t-s=4$ is given by
\begin{align*}
\m{E}_{2}^{2^{k+2}-4,2^{k+2}}=
&\Ind_{k+2}\mB_{2,k-1}^{-}\cdot\{u^{-}a_{\bar{\rho}_{k+1}-\sigma_{k+1}}^{2}u_{2\sigma_{k+1}}\res_{k+1}N_{1}^{k+2}r_{1}\} 
\oplus \\
& \bigoplus_{p\in P_{n,2}(C_{2^{k+1}})/C_{2^{n}}} \Ind_{k+1}\mB_{2,k-1}\cdot\{a_{\bar{\rho}_{k+1}-\sigma_{k+1}}^{2}u_{2\sigma_{k+1}}\bar{p}\} \oplus 
\\ & \bigoplus_{p\in P_{n,4}(C_{2^{k}})/C_{2^{n}}} \Ind_{k}\mB_{1,k-1}\cdot\{ a_{\bar{\rho}_{k}}^{4}\bar{p}\}, 
\end{align*}
for $1\leq k\leq n-2$. 

We also have
\[
\m{E}_{2}^{0,4}=
\Ind_{2}\mZ^{-}\cdot\{u^{-} u_{2\sigma_{1}} N_{1}^{2}r_{1}\} 
\oplus
\bigoplus_{p\in P_{n,2}(C_{2})/C_{2^{n}}}\Ind_{1}\mZ\cdot\{u_{2\sigma_{1}}\bar{p}\}. 
\]
All other Mackey functors for $t-s=4$ are zero.
\end{proposition}
\begin{proof}
The only indices which can contribute to $t-s=4$ come from 
\begin{enumerate}
\item $\mathcal P_{n,1}(C_{2^{k}})$ for $k\geq 2$, contributing in filtration $2^{k}-4$, 
\item $\mathcal P_{n,2}(C_{2^{k}})$ for all $k$ (though $k=1$ is special), contributing in filtration $2^{k+1}-4$, and 
\item $\mathcal P_{n,4}(C_{2^{k}})$ for all $k$, contributing in filtration $2^{k+2}-4$.
\end{enumerate}

Our summands correspond to each, in order, collecting them to be in a common filtration. The single case in which the third term does not contribute is the case of filtration zero.
%
\end{proof}

We note that there are no terms in degree $4$ arising from the non-orientable representations indexed by $\mathcal P_{n,3}(C_{2^{k}})$. This is because the representations are all of the form $3\rho_{k}$, and the restriction of this to $C_{2^{k-1}}$ is $6\rho_{k-1}$. In particular, the associated representation sphere is $5$-connected.

\subsection{The slice differentials and the $\m{E}_\infty$ term for $t-s=3$}

We are now in the situation where we can apply Theorem~\ref{thm:Hill:2009uoDifferentials} to determine all of the differentials on classes in $t-s=4$.

\begin{theorem}
The subMackey functors
\begin{align*}
\m{P}^{2^{k+2}-4,2^{k+2}} =
&\Ind_{k+2}(\mB_{1,k-1}^{\ast})^{-}\cdot\{u^{-}\res_{k}(a_{\bar{\rho}_{k+1}-\sigma_{k+1}}^{2}u_{2\sigma_{k+1}}\res_{k+1}N_{1}^{k+2}r_{1})\} 
\oplus \\
& \bigoplus_{p\in P_{n,2}(C_{2^{k+1}})/C_{2^{n}}} \Ind_{k+1}\mB_{1,k-1}^{\ast}\cdot\{\res_{k}(a_{\bar{\rho}_{k+1}-\sigma_{k+1}}^{2}u_{2\sigma_{k+1}}\bar{p})\} \oplus 
\\ & \bigoplus_{p\in P_{n,4}(C_{2^{k}})/C_{2^{n}}} \Ind_{k}\mB_{1,k-1}\cdot\{ a_{\bar{\rho}_{k}}^{4}\bar{p}\}, 
\end{align*}
for $1\leq k\leq n-2$ and
\[
\m{P}^{0,4}=\Ind_2(\mZ^{\ast})^{-}\cdot\{u^{-}\res_0(u_{2\sigma_1}N_1^2 r_1)\}\oplus \bigoplus_{p\in P_{n,2}(C_{2})/C_{2^{n}}}\Ind_{1}\mZ^{\ast}\cdot\{\res_0(u_{2\sigma_{1}}\bar{p})\}
\]
are composed entirely of permanent cycles.
\end{theorem}
\begin{proof}
The restrictions in question are generated by elements of the form $a_V\bar{p}$ (or just $p$ for the filtration zero elements), which is a permanent cycle. The Mackey functor generated by a permanent cycle is a collection of permanent cycles.
\end{proof}

This reduces our computation to understanding the differentials on the cokernel of these inclusions. However, Proposition~\ref{prop:SESMackey} shows that these cokernels are exceptionally simple: they are [possibly signed] induced up copies of $\mB_{1,k}$.

\begin{theorem}
The elements of $\m{E}_2^{2^{k+2}-4,2^{k+2}}/\m{P}^{2^{k+2}-4,2^{k+2}}$ survive to $\m{E}_{2^{k+1}+1}$. The $d_{2^{k+1}+1}$-differential on these classes is injective and the inclusion of a direct summand on the second summand.
\end{theorem}
\begin{proof}
The quotient in question has two kinds of summands: a single sign induced summand from $C_{2^{k+2}}$ and a collection of summands, indexed by orbits of $P_{2,n}(C_{2^{k+1}})$ induced from $C_{2^{k+1}}$. Maps out of all of these terms are determined by what happens to a generator (and in particular, to the restriction to $C_{2^{k+1}}$. Theorem~\ref{thm:Hill:2009uoDifferentials} shows that this is a $d_r$-cycle for $r<2^{k+1}+1$. That none of the classes are the target of a differential will follow from the second sentence, since a class supporting a differential cannot itself be the target of one.

The second part is slightly more involved, and we have to be more careful with the terms. Here we use the pullback. The $C_{2^{k+1}}$-Mackey functor $\mB_{1,k}$ is the pullback of the ``fixed point Mackey functor'' (here really an abelian group) under the absolute quotient map $C_{2^{k+1}}\to C_1$. Since induction commutes with pulling back, we conclude that
\[
\Ind_{k+1} \mB_{1,k}\cong \pi_{k+1}^\ast \m{\Ind_0^{n-k} \F_2}.
\]
In particular, the second summand in our quotient can be written in a much more transparent form:
\begin{multline*}
\bigoplus_{p\in P_{n,2}(C_{2^{k+1}})/C_{2^{n}}} \Ind_{k+1}\mB_{1,k}\cdot\{a_{\bar{\rho}_{k+1}-\sigma_{k+1}}^{2}u_{2\sigma_{k+1}}\bar{p}\}\cong \\
\pi_{k+1}^\ast \m{\F_2\cdot P_{n,2}(C_{2^{k+1}})}\cdot a_{\bar{\rho}_{k+1}-\sigma_{k+1}}^{2}u_{2\sigma_{k+1}},
\end{multline*}
where $\F_2\cdot P_{n,2}(C_{2^{k+1}})$ is the $\F_2[C_{2^n}]$-module with basis the $C_{2^n}$-set $P_{n,2}(C_{2^{k+1}})$. The differential in Theorem~\ref{thm:Hill:2009uoDifferentials} is the $d_{2^{k+1}+1}$-differential on $u_{2\sigma_{k+1}}$. The image of this differential lands in the Mackey functor
\[
\bigoplus_{\bar{p}\in P_{n,3}(C_{2^{k+1}})/C_{2^n}} \Ind_{k+1} \mB_{1,k}\cdot\{a_{\bar{\rho}_{k+1}}^3\bar{p}\},
\]
which, by the same analysis as above, is isomorphic to
\[
\pi_{k+1}^\ast\m{\F_2\cdot P_{n,3}(C_{2^{k+1}})}\cdot a_{\bar{\rho}_{k+1}}^3.
\]
The differential is then the pullback of the map
\[
\m{\F_2\cdot P_{n,2}(C_{2^{k+1}})}\to \m{\F_2\cdot P_{n,3}(C_{2^{k+1}})},
\]
induced by multiplication by $\fancydelta_1^{C_{2^{k+1}}}$ on the underlying $C_{2^n}/C_{2^{k+1}}$-modules. The underlying map of modules is obviously an inclusion, these modules also sit inside the polynomial ring as a subset of elements of a given degree. To show that this induces an injection on the corresponding Mackey functors, we actually need to show that it sits as a direct summand.

By assumption, the stabilizer of every element of the sets $P_{n,m}(C_{2^{k+1}})$ is $C_{2^{k+1}}$, and therefore the modules
\[
\F_2\cdot P_{n,m}(C_{2^{k+1}})
\]
are free $\F_2[C_{2^{n}}/C_{2^{k+1}}]$-modules. The long exact sequence in Tate cohomology, together with the classification of indecomposable $\F_2[C_{2^{n-k-1}}]$-modules shows that the cokernel of this map is also free. Direct sum decompositions of modules induce direct sum decompositions of the corresponding fixed point Mackey functors, and pulling back preserves direct sums. Thus the map induced by multiplication by $\fancydelta_1^{C_{2^{k+1}}}$ is the inclusion of a direct summand.

The summand $\mB_{1,k}^{-}\cdot\{ u^{-} a_{\bar{\rho}_{k+1}-\sigma_{k+1}}^2 u_{2\sigma_{k+1}} \res_{k+1}N_1^{k+2}r_1\}$ is treated similarly. The map is again induced by $\fancydelta_1^{C_{2^{k+1}}}$. Here, however, we need to ensure the the image of this map maps injectively into the complementary summand to the image of the differential on the simpler summand. However, this is immediate. By assumption, the stabilizer of $\res_{k+1}N_1^{k+2}r_1$ is $C_{2^{k+2}}$. In particular, neither it nor any linear combination with its translates under the group action lie in $\F_2\cdot P_{n,2}(C_{2^{k+1}})$. The image of the differential is the diagonal in the summand $\F_2\{(N_1^{k+1}r_1)^2\gamma N_1^{k+1}r_1, N_1^{k+1}r_1(\gamma N_1^{k+1} r_1)^2\}$, and since induction is exact, the differential is also injective here.
\end{proof}

Let 
\[
b_k=|P_{n,3}(C_{2^k})/C_{2^n}|-|P_{n,2}(C_{2^k})/C_{2^n}|-1.
\]
This is the number of summands complementary to the image of the differentials given by the previous theorem. This immediately gives the following theorem.

\begin{theorem}
As Mackey functors, the slice $\m{E}_{\infty}$-page for $t-s=3$ for $\Xi_{n}$ is given by
\[
\m{E}_{\infty}^{t-3,t}=\begin{cases}
\Ind_{k}\mB_{1,k-1} & t=2^{k}, n\leq k>1 \\
\Ind_{k}\mB_{1,k-2}^{\ast}\oplus b_{k}\Ind_{k-1}\mB_{1,k-2} & t=2^{k}+2^{k-1}, n\leq k>1 \\
0 & \text{ otherwise.}
\end{cases}
\]
\end{theorem}

\section{Some Mackey functor \texorpdfstring{$\Ext$}{Ext} groups and resolving the \texorpdfstring{$\m{\pi}_{3}\Xi_{n}$}{pi3} Extensions}\label{sec:ExtGroups}
To prove our main theorem, we need to resolve the extensions in $\m{\pi}_{3}$. Here is where the fact that these are Mackey functors is essential. Evaluating at the fixed points results in a huge collection of copies of $\Z/2\Z$, all of which could assemble to give quite large $2$-torsion groups. We will show that there are no possible extensions, beyond the ones given by Corollary~\ref{cor:Extension} using elementary $\Ext$ computations in the category of Mackey functors. To simplify notation, we will let $\Ext_{G}$ denote $\Ext$ in the category of $G$-Mackey functors. Similarly, $\Ext_{\mZ}$ will denote $\Ext$ in the category of modules over $\mZ$.

\subsection{\texorpdfstring{$\Ext$}{Ext} and the city}
Almost all of the Mackey functors which arise in our $E_{\infty}$ term are induced from proper subgroups of $C_{2^{n}}$. Since induction and coinduction agree, being both a left and right adjoint to the forgetful functor, we know that these preserve exact sequences and extensions. In particular, we have an obvious limiting case:

\begin{proposition}\label{prop:Induced}
If $\mM$ is a Mackey functor and $\m{N}=\Ind_{H}^{G}\m{N}'$, then
\[
\Ext_{G}^{\ast}(\m{N},\mM)\cong\Ext_{H}^{\ast}\big(\m{N}',\Res_{H}^{G}\mM\big).
\]

Similarly,
\[
\Ext_{G}^{\ast}(\mM,\m{N})\cong \Ext_{H}^{\ast}\big(\Res_{H}^{G}\mM,\m{N}'\big).
\]
\end{proposition}

\begin{corollary}\label{cor:NeededExtensions}
If $\mM$ is a Mackey functor such that $\Res_{H}^{G}\mM=0$, then for all $\m{N}=\Ind_{H}^{G}\m{N}'$,
\[
\Ext_{G}^{\ast}(\m{N},\mM)=\Ext_{G}^{\ast}(\mM,\m{N})=0.
\]

In particular, the direct sum is the unique extension, and there is are unique sections of the projections of $\mM\oplus\m{N}$ onto the factors.
\end{corollary}

This is exactly the situation with matches our $\m{E}_{\infty}$-term. Associated graded pieces of filtration at least $2^{k}-3$ always vanish when restricted to $C_{2^{k-2}}$, while those of filtration at most $3(2^{k}-1)$ are always induced from $C_{2^{k-1}}$.

We need one final direct $\Ext$-computation. Here, we need that $\m{\pi}_{0}\Xi=\mZ$. In particular, everything is a $\mZ$-module, and we can compute our $\Ext$ groups in that category instead.

\begin{lemma}\label{lem:DirectExtComp}
For any $k$, 
\[
\Hom_{\mZ}(\mB_{1,k-2}^{\ast},\mB_{1,k-1})=\Ext_{\mZ}^{1}(\mB_{1,k-2}^{\ast},\mB_{1,k-1})=0.
\]
\end{lemma}
\begin{proof}
We build the first few stages for a projective resolution for $\mB_{1,k-2}^{\ast}$. Explicit computation shows that the sequence
\[
\Ind_{k-2}\mZ\oplus \Ind_{k-1}\mZ\xrightarrow{[1_{k-2},(1+\gamma)_{k-1}]}\Ind_{k-1}\mZ\to \mB_{1,k-2}^{\ast}
\]
is exact, where the map labeled $[1_{k-2},(1+\gamma)_{k-1}]$ is the map induced by the identity for $C_{2^{k-2}}$ on the first factor (hence the $1_{k-2}$) and by multiplication by the element $1+\gamma\in \mZ[C_{2^{k}}/C_{2^{k-1}}]$ for $C_{2^{k-1}}$, and where the unlabeled map is induced from the canonical quotient $\mZ\to\mZ/2$ for $C_{2^{k-1}}$. In particular, the first two stages of a projective resolution for $\mB_{1,k-2}^{\ast}$ are induced from $C_{2^{k-1}}$, and we therefore conclude by Corollary~\ref{cor:NeededExtensions} that 
\[
\Hom_{\mZ}(\mB_{1,k-2}^{\ast},\mB_{1,k-1})=\Ext_{\mZ}^{1}(\mB_{1,k-2}^{\ast},\mB_{1,k-1})=0.\qedhere
\]
\end{proof}

\begin{lemma}\label{lem:SecondComputation}
For $C_{2^k}$ Mackey functors, we have
\[
\Ext_{\mZ}^{1}(\mB_{1,k-1},\mB_{1,k-1})=0.
\]
\end{lemma}
\begin{proof}
We can quickly form a projective resolution of $\mB_{1,k-1}$:
\[
\mZ\to\Ind_{k-1}\mZ\xrightarrow{1-\gamma}\Ind_{k-1}\mZ\to\mZ\to \mB_{1,k-1},
\]
where the map labeled $1-\gamma$ is induced from the Weyl action on $\Ind_{k-1}\mZ(C_{2^{k}}/C_{2^{k-1}})$, and where the unlabeled maps are induced from the identity when restricted to the appropriate subgroups. This lets us compute $\m{\Ext}(\mB_{1,k-1},\mB_{1,k-1})$:
\[
\mB_{1,k-1}\leftarrow 0\leftarrow 0\leftarrow \mB_{1,k-1},
\]
and the result follows.
\end{proof}

\subsection{Proof of Main Theorem}

\begin{maintheorem}
We have
\[
\m{\pi}_{3}\Xi_{n}=\bigoplus_{k\leq n} \Ind_{k}\mB_{2,k-2}\oplus b_{2}\Ind_{k}\mB_{1,k-1}.
\]
\end{maintheorem}

\begin{proof}

We prove by downward induction on $k$ that there are no non-trivial extensions amongst the Mackey functors of filtration at least $2^{k}-3$ except for the ones given by Corollary~\ref{cor:Extension}.

For $k=n$, the statement to prove is immediate. There is only room for one possible extension, and it is the one given by Corollary~\ref{cor:Extension}.

We therefore assume the result is true for $\ell>k$. In particular, this means that
\[
\mM_{k}=\bigoplus_{\ell>k}\Ind_{\ell} \mB_{2,\ell-2}\oplus \bigoplus_{\ell>k}b_{\ell}\Ind_{\ell-1} \mB_{1,\ell-2}
\]
is a subMackey functor of $\m{\pi}_{3}$. In particular, $\Res_{k-1}^{n}\mM_{k}=0$, and the restriction to $C_{2^{k}}$ is a sum of copies of $\mB_{1,k-1}$. 

To show that we simply add on the summands of filtration $2^{k}-3$ and $2^{k}+2^{k-1}-3$ (with their obvious extension), we show the vanishing of the relevant $\Ext$ groups vanish. We begin with filtration $2^{k}+2^{k-1}-3$:
\begin{multline*}
\Ext_{\mZ}\big(b_{k-1}\Ind_{k-1}\mB_{1,k-2}\oplus \Ind_{k}\mB_{1,k-2}^{\ast},\mM_{k}\big)\cong \\
b_{k-1}\Ext_{\mZ}\big(\Ind_{k-1}\mB_{1,k-2},\mM_{k}\big)\oplus
\Ext_{\mZ}\big(\Ind_{k}\mB_{1,k-2}^{\ast},\mM_{k}\big).
\end{multline*}
For the first summand, we note that we are mapping out of something induced from $C_{2^{k-1}}$. By Corollary~\ref{cor:NeededExtensions}, we conclude that these $\Ext$ groups are all zero. The only one here to compute is 
\[
\Ext_{\mZ}\big(\Ind_{k}\mB_{1,k-2}^{\ast},\mM_{k}\big)\cong\Ext_{\mZ}\big(\mB_{1,k-2}^{\ast},\Res_{k}\mM_{k}\big)\cong\bigoplus\Ext_{\mZ}\big(\mB_{1,k-2}^{\ast},\mB_{1,k-1}\big),
\]
but Lemma~\ref{lem:DirectExtComp} shows that this is zero.

Similarly, for filtration $2^{k}-3$, we need to show that 
\[
\Ext_{\mZ}\big(\Ind_{k}\mB_{1,k-1},\mM_{k}\big)=\Ext_{\mZ}\big(\Ind_{k}\mB_{1,k-1},\Ind_{k-1}\mB_{1,k-2}\big)=0.
\]
For the second piece, Corollary~\ref{cor:NeededExtensions} applied to the second factor shows the desired result. For the first, we need again only compute
\[
\Ext_{\mZ}^{1}\big(\mB_{1,k-1},\mB_{1,k-1}\big).
\]
Lemma~\ref{lem:SecondComputation} shows that this is zero.

Corollary~\ref{cor:Extension} shows that we have a non-trivial extension
\[
0\to\Ind_{k}\mB_{1,k-2}^{\ast}\to \Ind_{k}\mB_{2,k-2} \to \Ind_{k}\mB_{1,k-1}\to 0,
\]
as desired, and thus $\mM_{k-1}$ is a direct sum of $\mM_{k}$, $\Ind_{k}\mB_{2,k-2}$, and $b_{k}\Ind_{k}\mB_{1,k-2}$.
\end{proof}

\begin{corollary}
Multiplication by $4$ annihilates all elements in $\pi_{3}\XiO_{n}$.
\end{corollary}

\subsection{Identifying \texorpdfstring{$\nu$}{nu}}
The element $\nu$ is detected in $\pi_{3}\XiO_{n}$ for any $n\geq 1$, since it is detected in $\pi_{3}\XiO_{1}$ by work of Hu (alternatively, it is also detected in $\pi_{3} tmf_{0}(5)$) \cite{Hu:2001bu}, \cite{Behrens:vz}. Since the slice filtration is a filtration of Mackey functors, and since the element $\nu$ is detected in filtration $3$ in $\XiO_{1}$, we deduce the following.

\begin{proposition}
The element $\nu$ is represented by $\tr_{2}^{n}a_{\sigma_{2}}u_{\lambda} N_{1}^{2}\bar{r}_{1,2}+\tr_{1}^{n}a_{\sigma_1}^{3}\bar{r}_{3,2}$ modulo higher filtration.
\end{proposition}

In fact, $\nu$ is actually represented by a much more complicated (yet homogeneous) class. 
%
%

\begin{proposition}
The element $\nu$ is 
\[
\sum_{k=1}^{n-1} \tr_{k}^{n}a_{\sigma_{k}}u_{\lambda'} a_{\bar{\rho}'_{k}}N_{1}^{k}\bar{r}_{1,k} \mod a_{\bar{\rho}_{n}}N_{1}^{n}\bar{r}_{1}^{C_{2^{n}}}.
\]
\end{proposition}
\begin{proof}
We prove this by induction on $n$. The restriction to fixed points for smaller subgroups is explicitly determined by the above Mackey functor extensions, and by induction, we deduce that $\nu$ has the desired form through $k=n-1$.

To finish the proof, we use the geometric fixed points functor $\Phi^{C_{2^{n-1}}}$. Since our spectrum is a norm, we have an equivalence
\[
\Phi^{C_{2^{n-1}}}\Xi_n\simeq \Phi^{C_2} \Xi_2\simeq N_0^1MO.
\]
We therefore need only analyze the fate of $\nu$ in this last spectrum. However, classical, Singer construction type arguments show that $\nu$ (and also $\eta$ and $\sigma$!) is detected in the homotopy of this spectrum \cite{RognesPC}.

%
\end{proof}

\begin{remark}
An analysis identical to Theorem~\ref{thm:Differential} can also be used to show that $\nu$ is non-trivial on all powers of $a_{\lambda'}$. 
\end{remark}

\section{Conjectures and higher Hopf maps}\label{sec:HigherHopf}

Hu's work shows that not only are all of the Hopf invariant one classes detected in $\pi_{\ast}\XiO_{1}$, but also many other classes ($\epsilon$, etc) are non-zero \cite{Hu:2001bu}. In particular, this means that they are also non-zero in $\pi_{\ast}\XiO_{n}$ for all $n$. These classes often show up in a much higher slice filtration than might be expected. For example, for $\XiO_{1}$, the Hopf invariant one elements show up in bidegree $(2^{i}-1,2^{i}-1)$ (and in fact, the putative Hopf invariant one classes all persist here). As we saw above with $\nu$, as $n$ increases, the filtration of the lowest filtration class detected the element decreases. For $\nu$, by $n=2$, it was detected in filtration $1$. One of the way to interpret our main theorem is to note that the order stabilized at the same time the filtration did: for $C_{4}$, the maximum order for a torsion element is $4$, and for all larger $n$, $\nu$ still has this order.

Computational evidence suggests that $\sigma$ moves to filtration $1$ for $n=3$ (for $n=2$, it has filtration about $5$).  Here, the order is $8$, for largely the same arguments as for $\nu$ at $n=2$. This leads to two conjectures generalizing the above results.

\begin{conjecture}
The element $a_{\sigma_{n}} u_{\rho_{n}-\sigma_{n}}$ is a permanent cycle in the slice spectral sequence for $C_{2^{n}}$.
\end{conjecture}

For $\nu$, the corresponding result was needed for our above arguments, and for $\eta$, it is vacuously true. For $\sigma$, cursory computation shows that this class multiplied by the norm of $\bar{r}_{1,3}$ represents $\sigma$. Beyond this, however, it is unclear what to expect. In fact, conjecture is perhaps too strong a word here. There is a good chance that as $n$ increases, these classes support differentials reflecting the fact that not all of the classes $h_{j}$ survive the Adams spectral sequence. In this case, the differentials we see could be avatars of the Adams differential $d_{2}(h_{j})=h_{0}h_{j-1}^{2}$.

There is, however, some additional support for this conjecture unrelated to the homotopy groups of spheres. The proof of Theorem~\ref{thm:Cycle} goes through {\emph{mutatis mutandis}} to show the following.

\begin{proposition} 
For all $k<n-1$, the element
\[
a_{\sigma_n}u_{\lambda(2^k)}
\]
is a permanent cycle. 
\end{proposition}
More surprisingly, it is not difficult to show that $2u_{\lambda(2^k)}$ is always a permanent cycle, and this forces
\[
a_V u_{\lambda(2^k)}
\]
to be a non-zero permanent cycle whenever $V^{C_2^{k-1}}=V$.

A more ambitious, and less substantiated conjecture, reflects the nature of the order of an element and when it moves into the lowest filtration it ever achieves. For now, we can call this the ``stable filtration'' of an element.

\begin{conjecture}
The order of an element is fixed by the order it achieves when it moves into its stable filtration.
\end{conjecture}

\bibliographystyle{plain}
\bibliography{math}

\end{document}